\documentclass[a4paper]{amsart}
\usepackage[latin1]{inputenc}
\usepackage{amssymb}
\usepackage{amsmath}
\usepackage{mathrsfs}
\usepackage{eufrak}
\usepackage{amsthm}
\usepackage{amsfonts}
\usepackage{textcomp}
\usepackage{graphicx}
\usepackage[pdftex]{color}
\usepackage{paralist}
\usepackage[shortlabels]{enumitem}
\usepackage{hyperref}
\usepackage{comment}
\usepackage[arrow, matrix, curve]{xy}
\usepackage{tikz}

\newtheorem*{corollary*}{Corollary}
\newtheorem*{conjecture*}{Conjecture}
\newtheorem*{example*}{Example}
\newtheorem*{theorem*}{Theorem}
\newtheorem*{proposition*}{Proposition}

\newtheorem{theorem}{Theorem}[section]

\newtheorem{corollary}[theorem]{Corollary}
\newtheorem{lemma}[theorem]{Lemma}
\newtheorem{proposition}[theorem]{Proposition}

\newtheorem*{claim*}{Claim}

\theoremstyle{definition}
\newtheorem{definition}[theorem]{Definition}
\newtheorem{remark}[theorem]{Remark}
\newtheorem{example}[theorem]{Example}

\theoremstyle{remark}

\numberwithin{equation}{section}

\makeatletter
\renewcommand*\env@matrix[1][\
arraystretch]{%
  \edef\arraystretch{#1}%
  \hskip -\arraycolsep
  \let\@ifnextchar\new@ifnextchar
  \array{*\c@MaxMatrixCols c}}
\makeatother

\renewcommand{\mod}{\operatorname{mod}}

\newcommand{\Ext}{\operatorname{Ext}}

\newcommand{\Mirr}{\operatorname{M-irr}}
\newcommand{\Jirr}{\operatorname{J-irr}}
\newcommand{\domdim}{\operatorname{domdim}}

\newcommand{\cov}{\operatorname{cov}}
\newcommand{\cocov}{\operatorname{cocov}}
\newcommand{\gldim}{\operatorname{gldim}}
\newcommand{\End}{\operatorname{End}}

\newcommand{\pdim}{\operatorname{pdim}}
\newcommand{\idim}{\operatorname{idim}}
\newcommand{\row}{\operatorname{row}}
\newcommand{\gr}{\operatorname{grade}}
\newcommand{\cogr}{\operatorname{cograde}}
\newcommand{\Tr}{\operatorname{Tr}}
\newcommand{\Hom}{\operatorname{Hom}}

\renewcommand{\top}{\operatorname{\mathrm{top}}}

\newcommand{\soc}{\operatorname{\mathrm{soc}}}

\newcommand{\op}{\operatorname{op}}
\begin{document}

\title{Distributive lattices and Auslander regular algebras}
\date{\today}



\author{Osamu Iyama}
\address{Graduate School of Mathematical Sciences, University of Tokyo, 3-8-1 Komaba Meguro-ku Tokyo 153-8914, Japan}
\email{iyama@ms.u-tokyo.ac.jp}

\author{Ren\'{e} Marczinzik}
\address{Institute of algebra and number theory, University of Stuttgart, Pfaffenwaldring 57, 70569 Stuttgart, Germany}
\email{marczire@mathematik.uni-stuttgart.de}

\begin{abstract}
Let $L$ denote a finite lattice with at least two points and let $A$ denote the incidence algebra of $L$.
We prove that $L$ is distributive if and only if $A$ is an Auslander regular ring, which gives a homological characterisation of distributive lattices. In this case, $A$ has an explicit minimal injective coresolution, whose $i$-th term is given by the elements of $L$ covered by precisely $i$ elements.
We give a combinatorial formula of the Bass numbers of $A$. We apply our results to show that the order dimension of a distributive lattice $L$ coincides with the global dimension of the incidence algebra of $L$. Also we categorify the rowmotion bijection for distributive lattices using higher Auslander-Reiten translates of the simple modules.
\end{abstract}

\maketitle
\tableofcontents

\section*{Introduction}
The notion of Gorenstein rings is basic in algebra \cite{M,BH}. They are commutative Noetherian rings $R$ whose localization at prime ideals have finite injective dimension. One of the characterisations of Gorenstein rings due to Bass \cite{Ba} is that, for a minimal injective coresolution 
$$0 \rightarrow R \rightarrow I^0 \rightarrow I^1 \rightarrow \cdots \rightarrow I^i \rightarrow \cdots $$
of the $R$-module $R$ and each $i\ge0$, the term $I^i$ is a direct sum of the injective hull of $R/{\mathfrak p}$ for all prime ideals ${\mathfrak p}$ of height $i$.
In this case, $I^i$ has flat dimension $i$, and this property plays an important role in the study of non-commutative analogues of Gorenstein rings. A noetherian ring $A$, which is not necessarily commutative, satisfies the \emph{Auslander condition} if there exists an injective coresolution 
$$0 \rightarrow A \rightarrow I^0 \rightarrow I^1 \rightarrow \cdots \rightarrow I^i \rightarrow \cdots $$
of the right module $A$ such that the flat dimension of $I^i$ is bounded by $i$ for all $i \geq 0$. This condition is left-right symmetric, and there are other equivalent conditions \cite{FGR}. 
If $A$ has additionally finite injective dimension, $A$ is said to be \emph{Auslander--Gorenstein}, and if $A$ has finite global dimension, $A$ is said to be \emph{Auslander regular}.
These classes of algebras play important roles in various areas, including homological algebra \cite{AB,AR,AR2,GN1,GN2,H,HI,I,IwS,Mi,Ta}, non-commutative algebraic geometry \cite{ATV,L,VO,YZ}, analytic $D$-modules \cite{B}, combinatorial commutative algebra \cite{Y}, Auslander--Reiten theory \cite{ARS,AS,CIM,I2,I3,I4,IS} and tilting theory \cite{CS,IyZ,NRTZ,PS}. 
Many well studied classes of rings are Auslander regular, for example enveloping algebras of finite dimensional Lie algebras, Weyl algebras and the ring of $\mathbb{C}$-linear differential operators on an irreducible smooth subvariety of affine space, see \cite[Chapter 3]{VO} for proofs and more examples.
For survey articles on rings with the Auslander condition we refer to \cite{B2} and \cite{C}.

The incidence algebras of lattices are a central topic in combinatorics, see for example \cite{Ai} and \cite{St}. We refer to \cite{OS} for ring theory of the incidence algebras of posets, to \cite{S} for representation theory of incidence algebras, and to \cite{Hi} for commutative ring theory related to distributive lattices.
In this article we give a new link between the theory of distributive lattices and the theory of Auslander regular rings:
\begin{theorem*}[=Theorems \ref{homologicaldimensionsandausreg}, \ref{firstmainresult}]
Let $L$ be a finite lattice with incidence algebra $A$. Then $L$ is distributive if and only if $A$ is an Auslander regular algebra.
\end{theorem*}

More strongly, for a distributive lattice $L$, we construct explicitly a minimal injective coresolution
$$0 \rightarrow A \rightarrow I^0 \rightarrow I^1 \rightarrow \cdots \rightarrow I^i \rightarrow \cdots $$
of the right $A$-module $A$ such that $I^i$ is a direct sum of indecomposable injective $A$-modules $I(x)$ corresponding to elements $x\in L$ with $|\cov(x)|=i$ when $\cov(x)$ denotes the set of covers of $x$ in $L$, and we prove that $\pdim I(x)=|\cov(x)|$ holds for each $x\in L$ when $I(x)$ denotes the indecomposable injective module corresponding to $x$ (Theorem \ref{homologicaldimensionsandausreg}). In particular, each non-zero direct summand of $I^i$ has projective dimension precisely $i$, and hence $A$ belongs to a distinguished class of Auslander regular algebras called \emph{diagonal Auslander regular} (Definition \ref{define Auslander}).
Moreover, we give a combinatorial formula of the Bass numbers of $A$ (Corollary \ref{Bass number}).

One of the fundamental results in the theory of finite posets is Dilworth's Theorem
stating that the order dimension of a distributive lattice $L$ is equal to the width
of the poset of its join-irreducible elements, see
\cite[Theorem 1.2]{D} and \cite[Theorem 8.7, Chapter 2]{Tr}.  Our next main result gives a homological characterisation of the order dimension for distributive lattices:
\begin{theorem*}[=Theorem \ref{gldimdistlattice}]
Let $L$ be a finite distributive lattice with at least two elements. Then the order dimension of $L$ equals the global dimension of the incidence algebra of $L$.
\end{theorem*}
Two important applications of this result are that the global dimension of the incidence algebra of a distributive lattice is independent of the field, a result that is in general not true for arbitrary posets, and the property for a distributive lattice to be planar can be characterised purely homological, namely by having global dimension at most two.

We will categorify the rowmotion bijection for distributive lattices by using certain functors appearing in the representation theory of the incidence algebras. The \emph{rowmotion bijection} for a distributive lattice $L$, given as the set of order ideals of a poset $P$, is defined as $\row(x)$ being the order ideal generated by the minimal elements in $P\setminus x$ for an order ideal $x$ of $P$. The rowmotion bijection has appeared under many different names and is a useful tool in the study of combinatorial properties of posets arising for example in Lie theory, see for example the references \cite{Str} and \cite{TW} that also contain a historical background on the rowmotion bijection.

To categorify the rowmotion map we will use a canonical bijection, called the \emph{grade bijection}, between simple modules, that exists for any Auslander regular algebra $A$ by \cite[Theorem 2.10]{I}. This bijection associates to a simple $A$-module $S$ the simple $A$-module $\top(D \Ext_A^g(S,A))$, when $g=\gr S$ denotes the grade of $S$. It has an important representation theoretic meaning, see \ref{grade bijection for auslander algebra}.
For diagonal Auslander regular algebras, the grade bijection can be described by the \emph{$r$-Auslander-Reiten translate} \cite{I4}
\[\tau_r:=  D\Tr\Omega^{r-1}\]
for $r\ge1$ since we have $D \Ext_A^g(S,A) \cong \tau_g(S)$ for all simple modules $S$ with grade $g$, where $\Tr$ is the classical Auslander--Bridger transpose \cite{AB}. We summarize the main results on the rowmotion bijection for distributive lattices in the next theorem:

\begin{theorem*}[=Theorem \ref{rowmotion=AR}]
Let $A$ be the incidence algebra of a distributive lattice $L$. Then we have
\[\top(\tau_{p_x}(S_x))=S_{\row(x)}\ \mbox{ for each $x\in L$ and the number $p_x$ of covers of $x$.}\]
Moreover, each simple $A$-module $S_x$ satisfies $\gr S_x=\pdim S_x =|\cov(x)|$ (i.e.\ $S_x$ is perfect in the sense of \cite{BH}).
\end{theorem*}

\section{Preliminaries}
In this article all rings will be finite dimensional algebras over a field $K$ and all modules will be finitely generated right modules unless otherwise stated. 
\subsection{Preliminaries on representation theory and homological algebra}
We assume that the reader is familiar with the basics on representation theory of finite dimensional algebras and refer for example to the textbooks \cite{ARS} and \cite{SkoYam}. We remark that we adopt the convention of the book \cite{SkoYam} on path algebras to multiply arrows $\alpha_1 , ... , \alpha_n$ in a path algebra $KQ$ of a quiver $Q$ from left to right, that is $\alpha_1 \cdots \alpha_n$ denotes the path composed of $\alpha_1$ first and $\alpha_n$ last. We denote by $D=\Hom_K(-,K)$ the duality of $\mod A$ for a finite dimensional algebra $A$, $J$ the Jacobson radical of $A$ and $A^{\op}$ the opposite algebra of $A$.
Being a finite dimensional algebra implies that a finite dimensional module is projective if and only if it is flat and the flat dimension of such a module coincides with the projective dimension of the module, see for example \cite[Proposition 4.1.5]{W}. We denote by $\pdim M$ the projective dimension of a module $M$ and $\idim M$ denotes the injective dimension. 
The \emph{global dimension} of a finite dimensional algebra $A$ is defined by
\[\gldim A= \sup \{ \pdim M\mid M \in \mod A \}\]
and since $A$ has a duality this is also equal to $\sup \{ \idim M\mid M \in \mod A \}$. Clearly
\begin{equation}\label{gldim=suppdsimple}
\gldim A= \sup \{ \pdim S \mid S \ \text{is simple} \}
\end{equation}
holds, cf. \cite[Theorem 5.72 (Chapter 2)]{Lam}. 

\begin{definition}\label{define Auslander}
An algebra $A$ with minimal injective coresolution 
\begin{equation}\label{minimal injective resolution}
0 \rightarrow A \rightarrow I^0 \rightarrow I^1 \rightarrow \cdots \end{equation}
of the right $A$-module $A$ is said to be \emph{$n$-Gorenstein} for $n \geq 1$ if $\pdim I^i \leq i$ for all $i=0,...,n-1$. $A$ is called \emph{Auslander regular} if it is $n$-Gorenstein for all $n \geq 1$ and additionally has finite global dimension. An Auslander regular algebra $A$ is called \emph{diagonal Auslander regular} if $\pdim X=i$ for all $i\ge0$ when $X$ is a non-zero direct summand of $I^i$. 
\end{definition}

Thus an algebra $A$ is 1-Gorenstein if and only if the injective envelope $I^0$ of $A$ is projective.
The notion of being $n$-Gorenstein and Auslander regular is left-right symmetric. 
The notion of diagonal Auslander regular algebras was introduced in \cite{I3}.

The \emph{dominant dimension} $\domdim A$ of an algebra $A$ is a refined notion of $n$-Gorensteiness. It is defined as the infimum of $i\ge0$ such that $I^i$ in \eqref{minimal injective resolution} is not projective. Thus an algebra $A$ with $\domdim A\ge n$ is $n$-Gorenstein, and the converse is true for $n=1$.
By the Morita-Tachikawa correspondence an algebra $A$ has dominant dimension at least two if and only if $A \cong \End_B(M)$ for another algebra $B$ and a generator-cogenerator $M$ of $\mod B$. We refer to \cite{Ta} for more on the Morita-Tachikawa correspondence and the dominant dimension of algebras, and to \cite{ARS,AS,I4,IS} for their importance in representation theory.

We recall homological invariants closely related to $n$-Gorensteiness.
The \emph{grade} of a module $M$ over an algebra $A$ is defined as
\[\gr M:=\inf \{ \ell \geq 0 \mid \Ext_A^\ell(M,A) \neq 0 \}\]
and dually the \emph{cograde} is defined as
\[\cogr M:= \inf \{ \ell \geq 0 \mid \Ext_A^\ell(D(A),M) \neq 0 \}.\]
If a module $M$ has finite projective dimension, we always have $\gr M \leq \pdim M$, see for example \cite[Lemma 5.5 (Chapter VI)]{ARS}.
A module $M$ is often called \emph{perfect} if $\pdim M=\gr M< \infty$. The study of perfect modules is a classical topic in commutative algebra, e.g.\ \cite[Section 1.4]{BH}. 
The following is an important property of diagonal Auslander regular algebras.

\begin{proposition} \label{gradeproposition}
Let $A$ be an algebra such that $A$ and $A^{\op}$ are diagonal Auslander regular. Then each simple $A$-module $S$ satisfies $\gr S=\pdim S$ and $\cogr S=\idim S$.
\end{proposition}

\begin{proof}
We prove the second equality since the first one is dual.
First recall that for any module $M$ of finite injective dimension, $\idim M = \sup \{ i \geq 0 \mid \Ext_A^i(D(A),M) \neq 0 \}$, by the dual of \cite[Lemma 5.5 (Chapter VI)]{ARS}.
Since we assume that $A$ has finite global dimension as an Auslander regular algebra, any simple module has finite injective dimension and $\cogr S \leq \idim S$. 
Let $(P_i)$ be a minimal projective resolution of $D(A)$ and $S$ a simple $A$-module.
Then $\Ext_A^r(D(A),S) \neq 0$ if and only if the projective cover $T$ of $S$ is a direct summand of $P_r$, see for example \cite[Corollary 2.5.4]{Ben}.  
But since we assume that $A^{\op}$ is diagonal Auslander regular, every indecomposable term of $P_i$ has injective dimension equal to $i$. Thus $T$ appears in exactly one of the terms $P_i$ as a direct summand, let us say in term $P_r$.
Then $\Ext_A^r(D(A),S) \neq 0$ and $\Ext_A^t(D(A),S)=0$ for all $t \neq r$.
Thus $\cogr S=\idim S=r$.
\end{proof}

We denote by $\Omega^1(M)$ the \emph{first syzygy} of $M$, that is the kernel of the projective cover map of $M$. Dually, $\Omega^{-1}(M)$ denotes the \emph{first cosyzygy} of $M$, that is the cokernel of the injective envelope map of $M$. One defines inductively for $n \geq 2$: $\Omega^n(M):=\Omega^1(\Omega^{n-1}(M))$ and $\Omega^{-n}(M):=\Omega^{-1}(\Omega^{-(n-1)}(M))$. We set $\Omega^0(M)=M$.

The \emph{Auslander-Bridger transpose} $\Tr M$ of a module $M$ over a finite dimensional algebra with minimal projective presentation 
$$P_1 \xrightarrow{f} P_0 \rightarrow M \rightarrow 0$$
is defined as cokernel of the map $\Hom_A(f,A)$ so that we obtain the following exact sequence by applying the functor $\Hom_A(-,A)$ to the minimal projective presentation of $M$:
$$0 \rightarrow \Hom_A(M,A) \rightarrow \Hom_A(P_0,A) \xrightarrow{\Hom_A(f,A)} \Hom_A(P_1,A) \rightarrow \Tr M \rightarrow 0.$$
The \emph{$r$-Auslander-Reiten translate} is defined as $\tau_r(M):=D\Tr\Omega^{r-1}(M)$ for $r \geq 1$, where $\tau=\tau_1$ is the classical Auslander-Reiten translate.
We refer to \cite{ARS} and \cite{I4} for more on the $r$-Auslander-Reiten translates.
The next lemma follows immediately from the definition of $\Ext$ and the Auslander-Bridger transpose: 
\begin{lemma} \label{extlemma}
Let $A$ be a finite dimensional algebra and $M$ a module of finite projective dimension $p \geq 1$.
Then $\Ext_A^p(M,A) \cong \Tr\Omega^{p-1}(M)$ as left $A$-modules and $D\Ext_A^p(M,A) \cong \tau_p(M)$ as right $A$-modules.
\end{lemma}

\subsection{Preliminaries on lattice theory}
Lattices are a central topic in mathematics and there are several textbooks dedicated to the theory of lattices, see for example \cite{Bi} and \cite{G}.
For applications of lattices in combinatorics we refer to the textbooks \cite{Ai} and \cite{St}.
We assume that all posets in this article are finite. 

For a subset $X$ of a poset $P$ we denote by $\min(X)$ the set of minimal elements in $X$ and dually $\max(X)$ denotes the set of maximal elements in $X$. When $\min(X)$ consists of a unique element, we identify the set $\min(X)$ with this element and dually we identify $\max(X)$ with its unique element if there is a unique element in $\max(X)$. 
We call an element $M$ in a poset $P$ a \emph{global maximum} if $M \geq x$ for all $x \in P$ and we call an element $m$ a \emph{global minimum} if $m \leq x$ for all $x \in P$. Following \cite{G}, we call a poset \emph{bounded} if it has a global maximum and a global minimum. We denote a global maximum usually by $M$ and a global minimum by $m$. Recall that a poset is by definition a \emph{lattice} if any two elements $a$ and $b$ have a unique supremum, denoted by $a \lor b$, and a unique infimum, denoted by $a \land b$.
Recall that a lattice $L$ is called \emph{distributive} if one has $(x \land y) \lor (x \land z)= x \land (y \lor z)$ for all elements $x,y,z \in L$.
A \emph{sublattice} $T$ of a lattice $L$ is by definition a nonempty subset of $L$ such that $a,b \in T$ implies that $a \lor b \in T$ and $a \land b \in T$. For elements $s,t$ in a poset $P$ we say that $t$ \emph{covers} $s$ or $s$ is \emph{covered by} $t$ if $s < t$ and no element $u \in P$ satisfies $s < u < t$.
An element $p$ of a lattice $L$ is called \emph{join-irreducible} if $p$ is not the minimum $m$ of $L$ and for all $x,y \in L$ one has that $p= x \lor y$ implies that $p=x$ or $p=y$. \emph{Meet-irreducible} elements are defined dually. The \emph{Hasse diagram} or \emph{Hasse quiver} of a poset $P$ is by definition the quiver with vertices the elements of $P$ and a directed arrow from $i \in P$ to $j \in P$ iff $j$ covers $i$.
For an element $x \in L$ we denote by $\cov(x)$ the set of elements that cover $x$ and by $\cocov(x)$ the elements that are covered by $x$. Given two partial orders $\leq$ and $\leq^{*}$ on a finite set $X$, we say that $\leq^{*}$ is a \emph{linear extension} of $\leq$ exactly when $\leq^{*}$ is a total order and for every $x$ and $y$ in $X$, if $x \leq y$, then $x \leq^{*} y$.

The \emph{incidence algebra} of a poset $P$ is by definition the $K$-algebra over a field $K$ with $K$-basis given by $p_y^x$ for all elements $x,y \in P$ with $x \leq y$ and multiplication on basis elements given by $p_{x_2}^{x_1} p_{y_2}^{y_1}=p_{y_2}^{x_1}$ if $x_2=y_1$ and $p_{x_2}^{x_1} p_{y_2}^{y_1}=0$ else. In this article we will use the description of incidence algebras by quiver and relations, which is the most suitable for our purposes. 
That this description is equivalent to the usual definition of incidence algebras is easy to see and explained for example in \cite[Example 10 (Chapter 14.1)]{S}. Now let $P$ be a poset and $Q$ be the Hasse quiver of $P$ and $KQ$ its path algebra over the field $K$. Let $I$ the ideal of $KQ$ generated by all differences $w_1-w_2$ of two parallel paths $w_1$ and $w_2$ of length at least two that start at the same point and end at the same point. Then $A=KQ/I$ is isomorphic to the incidence algebra of $P$. The quiver $Q$ of the algebra is acyclic and thus the global dimension of $A$ is finite. When the points of a quiver algebra are $1,...,n$ then we denote by $e_i$ the primitive idempotents corresponding to $i$ and by $S_i$ the simple modules corresponding to $i$. We denote by $P(i)=e_i A$ the indecomposable projective modules corresponding to $i$ and by $I(i)=D(Ae_i)$ the indecomposable injective modules corresponding to $i$.
We denote by $p_{y}^{x}$ the arrow from $x$ to $y$ in the quiver of $Q$ when $x$ is covered by $y$. For general elements $a,b \in P$ we denote the unique path from $a$ to $b$ by $p_{b}^{a}$ if $a \leq b$ with $e_i=p_{i}^i$.
We give two examples:
\begin{example} \label{chainexample}
Let $P=C(n)$ be a chain with $n \geq 2$ elements numbered from 0 to $n-1$ such that 0 is the smallest and $n-1$ the largest element.
Then the quiver $Q$ of the incidence algebra $A$ looks as follows:
\[\begin{tikzpicture}[scale=0.9]

\node (v1) at (-20,22) {${\scriptstyle 0}$};

\node (v2) at (-18,22) {${\scriptstyle 1}$};

\node (v6) at (-16,22) {${\scriptstyle 2}$};

\node (v8) at (-14.5,22) {${\scriptstyle n-2}$};

\node (v10) at (-12.5,22) {${\scriptstyle n-1}$};

\node (v11) at (-15.5,22) {$\ \ \ {\scriptstyle \cdots}$};

\draw [->] (v1) edge node[above] {${\scriptstyle p_1^0}$} (v2);

\draw [->] (v2) edge node[above] {${\scriptstyle p_2^1}$} (v6);

\draw [->] (v8) edge node[above] {${\scriptstyle p_{n-1}^{n-2}}$} (v10);

\end{tikzpicture}\]
We have $A=KQ$ since the ideal $I$ is zero here.
It is elementary to check that $A$ is diagonal Auslander regular with $\gldim A=1$.
\end{example}

Note that $C(n)$ is actually the unique lattice such that the incidence algebra has global dimension equal to 1, since a quiver algebra of the form $KQ/I$ with acyclic $Q$ has global dimension at most equal to 1 if and only if $I=0$ for an admissible ideal $I$.
\begin{example}
Let $P$ be the Boolean lattice on two elements and $A$ its incidence algebra. Thus $P$ has 4 points, namely $\{1,2 \}, \{1 \} , \{2\}$ and the empty set $\emptyset$.
$Q$ looks as follows:
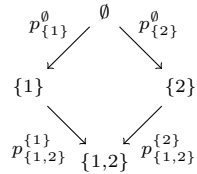
\[\begin{tikzpicture}[scale=1]
\node (A) at (0,0) {${\scriptstyle\emptyset}$};
\node (B) at (-1,-1) {${\scriptstyle\{1\}}$};
\node at (-2,-1) {};
\node (C) at (0,-2) {${\scriptstyle\{1,2\}}$};
\node (D) at (1,-1) {${\scriptstyle\{2\}}$};
\draw[->] (A) --node[left=.7em,above]{${\scriptstyle p_{ \{1\} }^{\emptyset}}$}(B);
\draw[->] (B) --node[left=1em,below]{${\scriptstyle p_{ \{1,2\} }^{\{1\}}}$}(C);
\draw[->] (A) --node[right=.7em,above]{${\scriptstyle p_{ \{2\} }^{\emptyset}}$}(D);
\draw[->] (D) --node[right=1em,below]{${\scriptstyle p_{ \{1,2\} }^{\{2\}}}$}(C);
\end{tikzpicture}\]
We have $I= \langle p_{ \{2 \} }^{\emptyset} p_{ \{1,2 \} }^{ \{2\} } - p_{ \{1 \} }^{\emptyset} p_{ \{1,2 \} }^{ \{1\} } \rangle$ and $A=KQ/I$.
It is elementary to check that $A$ is diagonal Auslander regular with $\gldim A=2$.

More generally, the incidence algebra $A$ of the Boolean lattice $P$ on $n$ elements is diagonal Auslander regular with $\gldim A=n$ as we will see later as a consequence of our main results.
\end{example}

The \emph{pentagon} is the lattice with a Hasse diagram of the following form: 
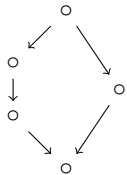
\[\begin{tikzpicture}[scale=0.7]
\node (A) at (0,0) {$\circ$};
\node (B) at (-1,-1) {$\circ$};
\node at (-2,-1.5) {};
\node (C) at (-1,-2) {$\circ$};
\node (D) at (0,-3) {$\circ$};
\node (E) at (1,-1.5) {$\circ$};
\draw[->] (A) -- (B);
\draw[->] (B) -- (C);
\draw[->] (C) -- (D);
\draw[->] (A) -- (E);
\draw[->] (E) -- (D);
\end{tikzpicture}\]
The \emph{diamond} is the lattice with a Hasse diagram of the following form:
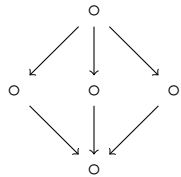
\[\begin{tikzpicture}[scale=0.7]

\node (v1) at (-18,23.5) {$\circ$};

\node (v2) at (-18,22) {$\circ$};

\node (v3) at (-18,20.5) {$\circ$};

\node (v4) at (-19.5,22) {$\circ$};

\node (v5) at (-16.5,22) {$\circ$};

\draw [->] (v1) edge (v2);

\draw [->] (v2) edge (v3);

\draw [->] (v1) edge (v4);

\draw [->] (v4) edge (v3);

\draw [->] (v1) edge (v5);

\draw [->] (v5) edge (v3);

\end{tikzpicture}\]

We will need the following characterisation of distributive lattices.

\begin{theorem}{ \cite[Theorem 1 (Section 7)]{G}}\label{distributivecriterion}
A lattice is distributive if and only if it does not contain a pentagon or a diamond as a sublattice.
\end{theorem}

A subset $\mathcal{I}$ of a poset $P$ is called an \emph{order ideal} of $P$ if for all $x,y \in P$ we have that $x \in \mathcal{I}$ and $y \leq x$ implies $y \in \mathcal{I}$. For a subset $A \subseteq P$ the set $\mathcal{I}(A):=\{ x \in P \mid x \leq a$ for some $a \in A \}$ is called the \emph{order ideal generated by $A$}. If $A= \{ a \}$, the set $\mathcal{I}(A)$ is called a \emph{principal order ideal} and we denote it also by $\mathcal{I}(a)$.
We define for a poset $P$ the lattice $\mathcal{O}_P$ as the set of all order ideals of $P$.
Note that this is a distributive lattice with the usual intersection and union of sets.

The next theorem is fundamental in the study of distributive lattices.

\begin{theorem}{\cite[Theorem 2.5 (Chapter II.1)]{Ai}} \label{posetidealstheorem}
Let $L$ be a distributive lattice and $P$ be the subposet of join-irreducible elements of $L$ viewed as a poset with the induced order of $L$.  
Then $L$ is isomorphic to $\mathcal{O}_P$.
\end{theorem}

We will also need the following well-known lemma for distributive lattices:

\begin{lemma}{\cite[Lemma 2.4 (Chapter II.1)]{Ai}} \label{irredundantdecomposition}
Let $L$ be a distributive lattice.
Then any element $a \in L$ with $a \neq m$ possesses a unique irredundant decomposition $a=y_1 \lor \cdots \lor y_r$ with pairwise different join-irreducible elements $y_i \in L$.

\end{lemma}

For two elements $x,y$ in a poset $P$, we denote by $[x,y]:= \{z \in P \mid x \leq z \leq y\}$ the interval between $x$ and $y$. For elements $v_1,...,v_n$ in a vector space $V$ we denote by $\langle v_1 ,..., v_n \rangle$ the vector subspace of $V$ generated by the elements $v_1 ,..., v_n$. For a vector $v$ in a vector space $V$, we denote by $v^*$ the dual vector in the dual vector space $V^*$. For a finite set $X$, we denote by $|X|$ the cardinality of $X$. For a subset $Y$ of $X$ we denote by $Y^c$ the complement of $Y$ in $X$.

\section{Antichain modules and global dimension of incidence algebras of lattices}
In this section we will give a projective resolution of certain modules in the incidence algebra $A$ of a lattice $L$ and use this to obtain upper bounds for the global dimension of $A$. 
When $m$ is the minimum of $L$ with incidence algebra $A$, the projective module $P(m)$ is the unique indecomposable projective-injective $A$-module and for every indecomposable projective module $P(i)$ there is a canonical monomorphism $P(i) \rightarrow P(m)$, this is easy to see and a special case of our later result in proposition \ref{domdimtheorem} that gives a classification of 1-Gorenstein incidence algebras for general posets.
Recall that an \emph{antichain} $C$ in a lattice $L$ is a subset $C$ of $L$ that consists of pairwise incomparable elements of $L$. We associate to an antichain $C=\{x_1,...,x_r\}$ the submodule
\[N_C:=\sum\limits_{i=1}^{r}{p_{x_i}^m A}\]
of the projective-injective indecomposable module $P(m)$, which is the right ideal of $A$ generated by the elements $p_{x_i}^m$ for $i=1,...,r$. The next proposition shows that in fact any submodule of $P(m)$ has the form $N_C$ for an antichain $C$.
\begin{proposition} \label{antichaincorrespondence}
Let $L$ be a lattice with incidence algebra $A$.
Then there is a bijection between antichains $C$ of $L$ and isomorphism classes of submodules of the indecomposable projective-injective module $P(m)$ given by associating to $C=\{x_1,...,x_r\}$ the submodule $\sum\limits_{i=1}^{r}{p_{x_i}^m A}$ of $P(m)$
\end{proposition}

\begin{proof}
Let $N$ be any non-zero submodule of $P(m)$. Then $N$ is finitely generated and let $x \in N$ be a generator of $N$. Then we can write $x= \sum\limits_{i=1}^{t}{\lambda_i p_{y_i}^m}$ for some non-zero field elements $\lambda_i$ and pairwise different $y_i \in L$.
Now since $N$ is a submodule, we have $x e_{y_i}=\lambda_i p_{y_i}^m \in N$ and thus $p_{y_i}^m \in N$. This means that $N$ is generated by elements of the form $p_{x_i}^m$ for some $x_i \in L$ and thus $N=\sum\limits_{i=1}^{r}{p_{x_i}^m A}$. Now we have $p_{x_i}^m A \subseteq p_{x_j}^m A$ when $x_j \leq x_i$ and thus we can assume that the set $\{x_1,...,x_r \}$ forms an antichains, which proves that every submodule of $P(m)$ is of the form $\sum\limits_{i=1}^{r}{p_{x_i}^m A}$ for an antichain $\{x_1,...,x_r \}$.
Now note that a $K$-basis of the submodule $\sum\limits_{i=1}^{r}{p_{x_i}^m A}$ is given by $\{p_u^m | u \geq x_i \ \text{for some} \ i=1,...,r \}$. Thus the submodules $N_{C}$ and $N_{C'}$ for two antichains $C$ and $C'$ are isomorphic if and only if $C=C'$, since the composition factors of the module $N_C$ are exactly the simple modules $S_{u}$ for $u \in L$ with $u \geq x_i$ for some $x_i \in C$. \qedhere

\end{proof}

For an antichain $C$ in a lattice $L$, we define the \emph{antichain module} associated to $C$ by
\[M_C:=P(m)/N_C.\]
Note that a module $M$ is an antichain module of the form $M_C$ for some antichain if and only if $M$ has projective cover $P(m)$ by the previous proposition. We now give a projective resolution for any antichain module.

\begin{theorem} \label{antichainmoduleresolution}
Let $L$ be a lattice with incidence algebra $A$. Let $C$ be an antichain of cardinality $\ell$ with associated module $M_C$.
Then $M_C$ has a projective resolution 
\[0\to P_\ell\to\cdots\to P_0\to M_C \to0\ \mbox{ with }\ P_0=P(m) \mbox{ and } \ P_r = \bigoplus\limits_{S\subseteq C,\ |S|=r}^{}P(\bigvee S) \ \mbox{ for }\ 1\le r\le\ell.\]
\end{theorem}

\begin{proof}
Let $f_0:P_0=P(m) \to M_C$ be the natural surjection.
For two elements $y\geq z$ in $L$, we denote by $\iota(y,z):P(y) \to P(z)$ the natural inclusion.
Moreover, we fix an arbitrary total ordering of elements in $C=\{x_1,...,x_\ell\}$, and for subsets $T\subset S\subseteq C$ with $T=S\setminus\{y\}$ for $y \in C$, let $\epsilon(S,T):=(-1)^{|\{z\in S\mid z<y\}|}$.
We construct a map $f_r:P_r\to P_{r-1}$ by
\begin{align*}
&f_r:=(f_{S,T})_{S,T}:P_r=\bigoplus_{S\subseteq C,\ |S|=r}P(\bigvee S)\to P_{r-1}= \bigoplus_{T\subseteq C,\ |T|=r-1}P(\bigvee T),\\
&\mbox{where }\ f_{S,T}:=\left\{\begin{array}{ll}\epsilon(S,T)\iota(\bigvee S,\bigvee T)&\mbox{if $S\supset T$,}\\
0&\mbox{else.}\end{array}\right.
\end{align*}
Now we check $f_{r-1}f_r=0$.  This is clear for $r=1$ since $\Hom_A(P(y),M_C)=0$ holds for any $y\in C$. Assume $r\ge2$. Take a direct summand $P(\bigvee S)$ of $P_r$ and $P(\bigvee U)$ of $P_{r-2}$. Clearly, the $(S,U)$-entry of $f_{r-1}f_r$ is non-zero only when there exist $y\neq z\in S$ such that $U=S\setminus\{y,z\}$. If there exist such $y,z$, then the $(S,U)$-entry of $f_{r-1}f_r$ is 
\[\left(P(\bigvee S)\xrightarrow{f_{S,S\setminus\{y\}}}P(\bigvee S\setminus\{y\})\xrightarrow{f_{S\setminus\{y\},U}}P(\bigvee U)\right)+\left(P(\bigvee S)\xrightarrow{f_{S,S\setminus\{z\}}}P(\bigvee S\setminus\{z\})\xrightarrow{f_{S\setminus\{z\},U}}P(\bigvee U)\right),\]
which is easily shown to be zero because of the sign $\epsilon$.

To check exactness of the sequence, it suffices to show that the sequence
\begin{equation}\label{multiply e_w}
0\to P_\ell e_w\xrightarrow{f_re_w}\cdots\xrightarrow{f_1e_w} P_0e_w\xrightarrow{f_0e_w} M_C e_w\to0
\end{equation}
is exact for each $w\in L$. Let $R_w:=\{y\in C \mid y\leq w\}$. Then we have
\[M_C e_w=\left\{\begin{array}{ll}k&R_w=\emptyset\\ 0&\mbox{else}\end{array}\right.\ \mbox{ and }\ P_re_w=\bigoplus_{S\subseteq R_w,\ |S|=r}k\ \mbox{ for }\ 0\le r.\]
Thus, if $R_w=\emptyset$, then \eqref{multiply e_w} is clearly exact. In the rest, assume $R_w\neq\emptyset$.
For subsets $T\subset S\subseteq R_w$ with $|S|=r=|T|+1$, the $(S,T)$-entry of $f_{S,T}e_w:P_re_w\to P_{r-1}e_w$ is $\epsilon(S,T)\cdot 1_k$. 
Thus the sequence \eqref{multiply e_w} is isomorphic to the tensor product of $|R_w|$ copies of $k\xrightarrow{1_k}k$ (i.e.\ the Koszul complex), and hence exact.
\end{proof}

We can use the previous theorem to obtain a projective resolution of all indecomposable injective modules for incidence algebras of lattices as the next result demonstrates:

\begin{corollary} \label{resolutioninjectivesandsimplesgenerallattice}
Let $L$ be a lattice with incidence algebra $A$ and $x \in L$. 
Then $\pdim I(x) \leq |\min([m,x]^c)|=:\ell$ holds, and $I(x)$ has a projective resolution 
\[0\to P_\ell\to\cdots\to P_0\to I(x) \to0\ \mbox{ with }\ P_0=P(m) \mbox{ and } \ P_r = \bigoplus\limits_{S\subseteq\min([m,x]^c),\ |S|=r}^{}P(\bigvee S) \ \mbox{ for }\ 1\le r\le\ell.\]
\end{corollary}

\begin{proof}
For $x=M$, we have $I(M) \cong P(m)$ and the statement is true. Now assume $x \neq M$, then we have $\top(I(x)) \cong D\soc(Ae_x) \cong S_m$ and thus $I(x)$ is an antichain module and has projective cover $f_0: P(m) \rightarrow I(x)$ given by $f_0(e_m a)= t^* a$, where $t:=p_x^m$ and $t^*$ denotes the dual vector in $D(Ae_x)=I(x)$.
The kernel $\ker(f_0)$ is given by $\ker(f_0)= \langle p_{q}^m | q \in [m,x]^c \rangle $ and thus is isomorphic to the submodule $N_C:=\sum\limits_{i=1}^{\ell}{p_{y_i}^m}$ of $P(m)$ where $\{y_1,...,y_\ell \}=\min([m,x]^c)$.
Thus $I(x) \cong M_C$ is an antichain modules corresponding to the antichain $C=\{y_1,...,y_\ell \}$ and the claim follows from \ref{antichainmoduleresolution}.
\end{proof}

We give an explicit example for the projective resolution of an indecomposable injective module:

\begin{example}
Let $L$ be a lattice with incidence algebra $A$ and let $x \in L$ such that $\min([m,x]^c)= \{y_1,y_2,y_3 \}$. 
Then a projective resolution of $I(x)$ is as follows:
$$0 \rightarrow P(y_1 \lor y_2 \lor y_3)\rightarrow P(y_1 \lor y_2)\oplus P(y_1 \lor y_3)\oplus P(y_2 \lor y_3) \rightarrow P(y_1)\oplus P(y_2)\oplus P(y_3) \rightarrow P(m) \rightarrow I(x) \rightarrow 0.$$
As an explicit example we can take $L$ to be the Boolean lattice of subsets of $\{1,2,3,4 \}$ and $x= \{4 \}$. Then $\min([\emptyset,\{4\}]^c)= \{ \{1\},\{2\},\{3\} \}$ and a projective resolution of $I(\{4\})$ is:
$$0 \rightarrow P(\{1,2,3 \})\rightarrow P(\{1,2\}) \oplus P(\{1,3\}) \oplus P(\{2,3\}) \rightarrow P(\{1\}) \oplus P(\{2\}) \oplus P(\{3\}) \rightarrow P(\emptyset) \rightarrow I(\{4\}) \rightarrow 0.$$
In this concrete example for the Boolean lattice the projective resolution is even minimal.
\end{example}

We also obtain a projective resolution of all simple modules over incidence algebras of lattices.

\begin{corollary}\label{resolutioninjectivesandsimplesgenerallattice2}
Let $L$ be a lattice with incidence algebra $A$ and $x \in L$. 
Then $\pdim S_x \leq |\cov(x)|=:\ell$ and $S_x$ has a projective resolution
\[0\to P_\ell\to\cdots\to P_0\to S_x\to0\ \mbox{ with }\ P_0=P_x \mbox{ and } \ P_r = \bigoplus\limits_{C\subseteq \cov(x),\ |C|=r}^{} P(\bigvee C)\ \mbox{ for }\ 1\le r\le\ell.\]
\end{corollary} 

\begin{proof}
It is clear that $S_M$ is projective. Let $x \neq M$ and let $L':=[x,M]$ be the interval from $x$ to $M$ in $L$ and note that $L'$ is also a lattice. Let $A'$ denote the incidence algebra of $L'$ and let $I^{L'}(i)$ denote the indecomposable injective $A'$-modules corresponding to the points $i \in L'$.
Note that the $A'$-modules $I^{L'}(i)$ are also $A$-modules and we have $I^{L'}(x) \cong S_x$ as $A$-modules.
A projective resolution of $I^{L'}(x)$ as an $A'$-module also gives a projective resolution as an $A$-module and thus we immediately obtain the result from \ref{resolutioninjectivesandsimplesgenerallattice}
 since in $L'$ the element $x$ is the minimum of $L'$ and thus $\min([x,x]_{L'}^c)$ (here $[a,b]_{L'}$ denotes the interval in the distributive lattice $L'$) is exactly the set of covers $\cov(x)$ of $x$ in $L$.  
\end{proof}

Now we give an upper bound of global dimension in terms of lattice structure.

\begin{theorem}
Let $L$ be a lattice with incidence algebra $A$.
Then
\[\gldim(A) \leq \max_{x\in L}|\cov(x)|.\]
\end{theorem}

\begin{proof}
The claim follows from \eqref{gldim=suppdsimple} and \ref{resolutioninjectivesandsimplesgenerallattice2}.
\end{proof}

The following example shows that the equality in the previous theorem is not true in general.

\begin{example}
Let $L$ be the diamond lattice with minimum $x=m$ and incidence algebra $A$. $S_m$ has projective dimension two and the global dimension of $A$ is equal to two, while $m$ has three covers.
\end{example} 

\section{Homological properties of incidence algebras of distributive lattices}

\subsection{Auslander regularity and Bass numbers}
In this section we prove that the incidence algebras of distributive lattices are Auslander regular algebras and give applications. 
We assume in the following that all lattices have at least two elements. 
We denote by $\Mirr(L)$ the set of meet-irreducible elements of a lattice $L$ and by $\Jirr(L)$ the set of join-irreducible elements of $L$.

\begin{proposition} \label{maxboundproposition}
Let $L$ be a distributive lattice.
\begin{enumerate}
\item[\rm(1)]
\begin{enumerate}[\rm(i)]
\item For $x\in L$, we have $\min([m,x]^c)\subseteq\Jirr(L)$.
\item There is a bijection $\alpha_x: \min([m,x]^c) \rightarrow \cov(x)$ given by $\alpha_x(z)= z \lor x$.
\item For each subset $S\subseteq\min([m,x]^c)$, let $y:=\bigvee S$. Then $| \max([y,M]^c) | = |S|$.
\end{enumerate}
\item[\rm(2)] 
\begin{enumerate}[\rm(i)]
\item For $y\in L$, we have $\max([y,M]^c)\subseteq\Mirr(L)$.
\item There is a bijection $\beta_y : \max([y,M]^c) \rightarrow \cocov(y)$ given by $\beta_y(z) = z \land y$.
\item For each subset $S\subseteq\max([y,M]^c)$, let $x:=\bigwedge S$. Then $| \min([m,x]^c) | = |S|$.
\end{enumerate}
\end{enumerate}
\end{proposition}

\begin{proof}
We prove (1), the proof of (2) is dual. We use \ref{posetidealstheorem} to represent $L=\mathcal{O}_P$ as the distributive lattice of order ideals of a poset $P$. For $p\in P$, we denote by
\[\mathcal{I}(p):=\{q\in P\mid q\le p\}\subseteq P\ \mbox{ \ the order ideal of $P$ generated by $p$ and }\ \mathcal{J}(p):=\{q\in P\mid p\not\le q\}\subseteq P.\]

(i) We regard $x\in L$ as an order ideal $x \subseteq P$.
For each $z\in\min([m,x]^c)$, take a minimal element $p$ in $z\setminus x$. Then $z=\mathcal{I}(p)$ holds since $[m,x]^c\ni I(p)\subseteq z$ and $z$ is minimal in $[m,x]^c$.
In particular, $z=\mathcal{I}(p)$ is join-irreducible since it covers a unique element $\mathcal{I}(p)\setminus\{p\}$.

(ii) By the above argument, there is a bijection $\min(P\setminus x)\to\min([m,x]^c)$ given by $p\mapsto \mathcal{I}(p)$.
On the other hand, elements in $\cov(x)$ are the order ideals $z$ of $P$ such that $z\supset x$ and $|z\setminus x|=1$. Thus there is a bijection $\min(P\setminus x)\to\cov(x)$ given by $p\mapsto x\cup\{p\}$. Since $x\cup\{p\}=x\vee \mathcal{I}(p)$, the assertion follows.

(iii) Let $S=\{I(p_i)\mid1\le i\le r\}$. Then 
\[y=\bigvee S=\bigcup_{i=1}^r \mathcal{I}(p_i).\]
We prove $\max([y,M]^c)=\{\mathcal{J}(p_i)\mid 1\le i\le r\}$. Then $| \max([y,M]^c) |=r=|S|$ holds. Clearly $\mathcal{J}(p_i)\in\max([y,M]^c)$ holds. Conversely, for each $z\in\max([y,M]^c)$, there exists $1\le i\le r$ such that $p_i\notin z$. Then $z=\mathcal{J}(p_i)$ holds since $[y,M]^c\ni \mathcal{J}(p_i)\supseteq z$ and $z$ is maximal in $[y,M]^c$.
\end{proof}


Now we state our main result.

\begin{theorem} \label{homologicaldimensionsandausreg}
Let $L$ be a distributive lattice with incidence algebra $A$ and $x,y\in L$.
\begin{enumerate}
\item[\rm(1)] $A$ is diagonal Auslander regular and has a minimal injective coresolution
$$0 \rightarrow A \rightarrow I^0 \rightarrow I^1 \rightarrow \cdots \rightarrow I^i \rightarrow \cdots $$
such that $I^i$ is a direct sum of copies of $I(x)$ for elements $x\in L$ with $|\cov(x)|=i$.
\item[\rm(2)] 
$\pdim I(x)=|\min([m,x]^c)|= |\cov(x)|=:\ell$ holds, and $I(x)$ has a minimal projective resolution 
\[0\to P_\ell\to\cdots\to P_0\to I(x) \to0\ \mbox{ with }\ P_0=P(m) \mbox{ and } \ P_r = \bigoplus\limits_{S\subseteq\min([m,x]^c),\ |S|=r}^{}P(\bigvee S) \ \mbox{ for }\ 1\le r\le\ell.\]
\item[\rm(3)] $\idim P(y)=|\max([y,M]^c)|= |\cocov(y)|=:\ell$ holds, and $P(y)$ has a minimal injective coresolution
\[0\to P(y)\to I^0\to\cdots\to I^\ell\to0\ \mbox{ with }\ I^0=I(M) \mbox{ and } \ I^r = \bigoplus\limits_{S\subseteq\max([y,M]^c),\ |S|=r}^{}I(\bigwedge S) \ \mbox{ for }\ 1\le r\le\ell.\]
\end{enumerate}
\end{theorem}

\begin{proof}
(2) The result is clear for $x=M$ and thus assume $x \neq M$. The projective resolution of the given form exists by \ref{resolutioninjectivesandsimplesgenerallattice} and we just have to show that it is minimal. By (1) of \ref{maxboundproposition} we know that each element in $\min([m,x]^c)$ is join-irreducible.
The elements $\bigvee S \in L$ for $S\subseteq\min([m,x]^c)$ are pairwise different since elements in a distributive lattice have a unique irredundant join-irreducible representation by \ref{irredundantdecomposition}. Thus the terms $P(\bigvee S)$ are pairwise non-isomorphic.
Since the terms $P(\bigvee S)$ are pairwise non-isomorphic, the projective resolution is minimal. Therefore, the projective dimension of $I(x)$ is equal to $|\min([m,x]^c)|$ and $|\min([m,x]^c)|= |\cov(x)|$ by \ref{maxboundproposition}.

(3) This follows from (2) by using duality since the opposite algebra $A^{\op}$ is the incidence algebra of the opposite distributive lattice of $L$.

(1) Since $A= \bigoplus_{y \in L}^{}{P(y)}$, we can use (3) to obtain a minimal injective coresolution $(I^r)$ of $A$. Each indecomposable direct summand of $I^r$ can be written as $I(x)$ for $x:=\bigwedge S$, where $S\subseteq\max([y,M]^c)$ with $|S|=r$ by (3). By (2) and Proposition \ref{maxboundproposition}(2)(iii), we have
\[\pdim I(x)=|\cov(x)|=|\min([m,x]^c)|=|S|=r.\]

Thus the assertions follow.
\end{proof}

We refer to \cite{Y} for a similar type of results for the incidence algebras of the posets associated with simplicial complexes.

Our previous results imply the following corollary, which makes Proposition \ref{gradeproposition} more explicit.

\begin{corollary} \label{corollaryperfectdistlattice}
Let $A$ be the incidence algebra of a distributive lattice $L$. Then the simple $A$-module $S_x$ corresponding to $x\in L$ satisfies
\[\gr S_x=\pdim S_x=|\cov(x)|\ \mbox{ and }\ \cogr S_x=\idim S_x=|\cocov(x)|.\]
\end{corollary}
\begin{proof}
We prove the first equalities since the second equalities follow dually.
Since $A$ is diagonal Auslander regular, we know by \ref{gradeproposition} that $\gr S_x=\pdim S_x$. By \ref{homologicaldimensionsandausreg} (2) the projective resolution of the indecomposable injective $A$-modules are minimal when $L$ is distributive and thus also the projective resolution of the simple modules $S_x$ in \ref{resolutioninjectivesandsimplesgenerallattice2} are minimal when $L$ is distributive as follows by the proof of \ref{resolutioninjectivesandsimplesgenerallattice2}.
Thus $\pdim S_x=|\cov(x)|$  \qedhere
\end{proof}

The equality $\gr S_x=\pdim S_x$ for simple modules $S_x$ is not true for general lattices as the next examples shows:
\begin{example}
Let $L$ be the lattice with the following Hasse quiver:
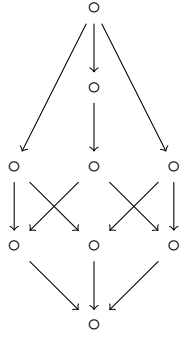
\[\begin{tikzpicture}[scale=0.7]
\node (v0) at (-19.5,28.5) {$\circ$};
\node (v1) at (-19.5,27) {$\circ$};
\node (v2) at (-18,25.5) {$\circ$};
\node (v3) at (-19.5,25.5) {$\circ$};
\node (v4) at (-21,25.5) {$\circ$};
\node (v5) at (-18,24) {$\circ$};
\node (v6) at (-19.5,24) {$\circ$};
\node (v7) at (-21,24) {$\circ$};
\node (v8) at (-19.5,22.5) {$\circ$};
\draw [->] (v0) edge (v1);
\draw [->] (v0) edge (v2);
\draw [->] (v0) edge (v4);
\draw [->] (v1) edge (v3);
\draw [->] (v2) edge (v5);
\draw [->] (v2) edge (v6);
\draw [->] (v3) edge (v5);
\draw [->] (v3) edge (v7);
\draw [->] (v4) edge (v6);
\draw [->] (v4) edge (v7);
\draw [->] (v5) edge (v8);
\draw [->] (v6) edge (v8);
\draw [->] (v7) edge (v8);
\end{tikzpicture}\]
The module $S_0$ has $\gr S_0=2$ but $\pdim S_0=3$.
\end{example}

As in the case of commutative rings \cite{BH} (see also \cite{GN1}), we study the following notion.

\begin{definition}\cite{S2}
Let $L$ be a poset and $A$ the incidence algebra of $L$. For $M\in\mod A$, let
\[0\to M\to I^0\to I^1\to\cdots\]
be a minimal injective resolution. For $x\in L$ and $i\ge0$, we define the \emph{$i$-th Bass number} $\mu^i(x,M)$ of $M$ as the multiplicity of $I(x)$ in $I^i$.
\end{definition}

We have the following description of the Bass numbers of $A$

\begin{corollary}\label{Bass number}
Let $L$ be the distributive lattice of the order ideals of a poset $P$, and $A$ the incidence algebra of $L$. For $x,y\in L$ and $i\ge0$, we have
\begin{align*}
\mu^i(x,P(y))&=\left\{\begin{array}{ll}
1&\mbox{if $i=|\cov(x)|$ and $\min(P\setminus x)\subseteq \max(y)$,}\\
0&\mbox{else.}
\end{array}\right.\\
\mu^i(x,A)&=\left\{\begin{array}{ll}
|\{y\in L\mid \min(P\setminus x)\subseteq \max(y)\}|&\mbox{if $i=|\cov(x)|$,}\\
0&\mbox{else.}
\end{array}\right.
\end{align*}
\end{corollary}

\begin{proof}
It suffices to prove the first equality. By Theorem \ref{homologicaldimensionsandausreg}, the following conditions are equivalent.
\begin{enumerate}[\rm(i)]
\item $\mu^i(x,P(y))\neq0$.
\item $\mu^i(x,P(y))=1$.
\item There exists $S\subseteq\max([y,M]^c)$ such that $|S|=i$ and $\bigwedge S=x$.
\end{enumerate}
The equality $\bigwedge S=x$ in (iii) is equivalent to $S=\{\mathcal{J}(p)\mid p\in\min(P\setminus x)\}$ for $\mathcal{J}(p):=\{q\in P\mid p\not\le q\}$. In this case, $|S|=|\cov(x)|$ holds. Moreover, there is a bijection $\max(y)\to\max([y,M]^c)$ given by $p\mapsto \mathcal{J}(p)$. Thus (iii) is equivalent to the following.
\begin{enumerate}
\item[\rm(iv)] $i=|\cov(x)|$ and $\min(P\setminus x)\subseteq\max(y)$.
\end{enumerate}
Thus the first equality holds.
\end{proof}

\subsection{Global dimension and order dimension}

We use our results to calculate the global dimension of the incidence algebra of a distributive lattice.
First we start with the general definition of order dimension for posets. 
\begin{definition}
The \emph{order dimension} of a poset $P$ is by definition the minimal number $t$ such that there exist $t$ linear extensions of $P$ whose intersection is equal to $P$.
\end{definition}

We refer for example to \cite[Chapter 6]{CLM} or the book \cite{Tr} for more characterisations and properties of the order dimension. The next result is the classical theorem of Dilworth.

\begin{theorem}{\cite[Theorem 1.2]{D}} \label{Dilworththeorem}
Let $L$ be a distributive lattice. Then the order dimension of $L$ coincides with $\max \{ |\cov(x)| \mid x \in L \}$.
\end{theorem}

The theorem of Dilworth can be used to calculate the global dimension for incidence algebras of distributive lattices as the next theorem shows.

\begin{theorem} \label{gldimdistlattice}
Let $L$ be a distributive lattice with at least two elements.
Then the global dimension of the incidence algebra of $L$ is equal to the order dimension of $L$.
\end{theorem}

\begin{proof}
By \ref{Dilworththeorem}, the order dimension of $L$ coincides with $\max \{ |\cov(x)| | x \in L \}$. But by \ref{corollaryperfectdistlattice} $\pdim S_x=|\cov(x)|$ and thus: 
\[\max \{ |\cov(x)| \mid x \in L \} = \max \{ \pdim S_x \mid x \in L \}=\gldim A.\qedhere\]
\end{proof}

We give two corollaries of the previous theorem.

\begin{corollary}
The global dimension of the incidence algebra of a distributive lattice over a field $K$ is independent of the field.

\end{corollary}
\begin{proof}
This follows immediately from \ref{gldimdistlattice}, since the order dimension of the lattice does not depend on the field. \qedhere

\end{proof}
Note that the previous corollary might be surprising since in general the global dimension of the incidence algebra of a poset can depend on the field, see for example \cite[Proposition 2.3]{IgZ}.

Recall that a poset $P$ is called \emph{planar} if its Hasse quiver can be drawn in the plane without intersections of arrows.
Our next application reveals the surprising fact that being planar can be characterised purely homological for distributive lattices.

\begin{corollary}
Let $L$ be a distributive lattice with incidence algebra $A$.
Then $L$ is planar if and only if the global dimension of $A$ is at most two.
\end{corollary}
\begin{proof}
The equivalence of (1) and (2) follows immediately by \ref{gldimdistlattice} combined with the fact that a lattice is planar if and only if its order dimension is at most 2, see for example \cite[Theorem 5.1 (Chapter 3)]{Tr}. \qedhere
\end{proof}

We give an example that shows that in general the global dimension of the incidence algebra of a (non-distributive) lattice does not coincide with the order dimension of this lattice.
\begin{example} \label{latticeexample10points}
Let $L$ be the lattice with the following Hasse diagram:
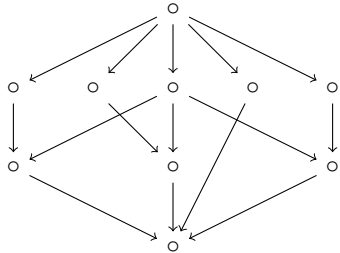
\[\begin{tikzpicture}[scale=0.7]
\node (v1) at (-19.5,27) {$\circ$};

\node (v2) at (-19.5,25.5) {$\circ$};

\node (v3) at (-21,25.5) {$\circ$};

\node (v4) at (-19.5,22.5) {$\circ$};

\node (v5) at (-16.5,24) {$\circ$};

\node (v6) at (-22.5,24) {$\circ$};

\node (v7) at (-18,25.5) {$\circ$};

\node (v8) at (-16.5,25.5) {$\circ$};

\node (v9) at (-22.5,25.5) {$\circ$};

\node (v0) at (-19.5,24) {$\circ$};

\draw [->] (v1) edge (v3);

\draw [->] (v1) edge (v2);

\draw [->] (v1) edge (v9);

\draw [->] (v1) edge (v8);

\draw [->] (v1) edge (v7);

\draw [->] (v3) edge (v0);

\draw [->] (v2) edge (v0);

\draw [->] (v2) edge (v6);

\draw [->] (v2) edge (v5);

\draw [->] (v9) edge (v6);

\draw [->] (v8) edge (v5);

\draw [->] (v0) edge (v4);

\draw [->] (v6) edge (v4);

\draw [->] (v5) edge (v4);

\draw [->] (v7) edge (v4);

\end{tikzpicture}\]
Then $L$ has order dimension 3 but the global dimension of the incidence algebra of $L$ is equal to two.
\end{example}

\section{A categorification of the rowmotion bijection for distributive lattices}
We saw that the incidence algebra of distributive lattices are diagonal Auslander regular algebras. In this section we apply this result to categorify the rowmotion bijection for distributive lattices.
Let $L$ be a distributive lattice given as the set of order ideals of a poset $P$. Then the \emph{rowmotion bijection} $\row$ of $L$ is given by
\begin{equation}\label{define row}
\row(x)=\bigcup_{p\in\min(P\setminus x)}\mathcal{I}(p)
\end{equation}
for an order ideal $x$ of $P$.
This is a bijection on $L$ and this bijection has several applications for the combinatorial study of posets in Lie theory, we refer for example to \cite{TW} and \cite{S} for more on rowmotion for lattices.

Our key observation is the next theorem which shows that any Auslander regular algebra $A$ has a canonical bijection between simple $A$-modules and simple $A^{\op}$-modules.

\begin{theorem}{\cite[Theorem 2.10]{I}} \label{diagausregbijection}
Let $A$ be an Auslander regular algebra.
Then the map $S \mapsto\soc\Ext_A^g(S,A)$ for $g=\gr  S$ gives a bijection between the simple $A$-modules and the simple $A^{\op}$-modules. Moreover it preserves the grade.
\end{theorem}

Since we prefer to work with right modules, we use the dual of the modules $\soc\Ext_A^g(S,A)$ in the following. Note that for a general module $M$, we have $D\soc M \cong \top( DM)$.
We are ready to introduce the following notion.

\begin{definition}
Let $A$ be an Auslander regular algebra. Using Theorem \ref{diagausregbijection}, we obtain a permutation $g_A$ of the simple $A$-modules, which we call the \emph{grade bijection}, given by
\[g_A(S):=\top(D\Ext_A^g(S,A))\ \mbox{ where }\ g=\gr S.\] 
\end{definition}


When $A$ is an Auslander algebra, the grade bijection has the following representation theoretic meaning.

\begin{remark}\label{grade bijection for auslander algebra}
Let $A$ be the Auslander algebra of a representation-finite algebra $B$.
Then the simple $A$-modules $S_M$ correspond to the indecomposable $B$-modules $M$. Then $g=\gr S_M$ is $2$ if $M$ is non-projective, and $0$ if $M$ is projective. Moreover, we have
\[\top(D \Ext_A^g(S_M,A)) \cong \begin{cases}
S_{\tau(M)}&\mbox{ if $M$ is non-projective,}\\
S_{\nu(M)}&\mbox{ when $M$ is projective.}
\end{cases}\]
where $\tau(M)$ denotes the Auslander-Reiten translate of $M$ and $\nu(M)$ the Nakayama functor applied to $M$.
Thus the grade bijection of $A$ gives a realisation of the Auslander-Reiten translate and Nakayama functor of $B$. A proof of those statements in the more general context of higher Auslander algebras can be found in the forthcoming article \cite{MTY}.
\end{remark}

Now we assume that $A$ and $A^{\op}$ are diagonal Auslander regular. Then by Proposition \ref{gradeproposition}, we have $\gr S=\pdim S$ for each simple $A$-module $S$, and by Lemma \ref{extlemma}, we have
\begin{equation}\label{DExt is AR}
g_A(S)=\top(D\Ext_A^p(S,A)) \cong \top(\tau_p(S))\ \mbox{ for }\ p=\pdim S,
\end{equation}
where $\tau_p$ is the $p$-Auslander-Reiten translate.



Our main result shows that the grade bijection $g_A$ gives a categorification of the rowmotion bijection $\row$.

\begin{theorem}  \label{rowmotion=AR}
Let $A$ be the incidence algebra of a distributive lattice $L$. Then we have
\[g_A(S_x)=S_{\row(x)}\ \mbox{ for each }\ x\in L.\]
\end{theorem}

\begin{proof}
Fix $x\in L$ and let $g:=\gr S_x=|\cov(x)|$.
Since $g_A(S_x)$ is simple, it suffices to show that $S_{\row(x)}$ is a factor $A$-module of $D\Ext^g_A(S_x,A)$, or equivalently, $S^{\op}_{\row(x)}$ is a sub $A^{\op}$-module of $\Ext^g_A(S_x,A)$.
To prove this, it suffices to prove the following two claims.
\begin{enumerate}[\rm(1)]
\item $\Ext^g_A(S_x,P(\row(x)))\neq0$.
\item $\Ext^g_A(S_x,P(y))=0$ for each $y\in L$ with $y<\row(x)$.
\end{enumerate}
On the other hand, the following conditions are equivalent for $y\in L$ by Corollary \ref{Bass number}.
\begin{enumerate}[\rm(i)]
\item $\Ext^g_A(S_x,P(y))\neq0$.
\item $\min(P\setminus x)\subseteq\max(y)$.
\end{enumerate}
By \eqref{define row}, $\min(P\setminus x)=\max(\row(x))$ holds, and any $y$ satisfying (ii) satisfies $\row(x)\le y$. Thus the claims (1) and (2) hold.
\end{proof}

Combining our results, we obtain an equality $|\cocov(\row(x))|=|\cov(x)|$ for each $x\in L$:
\[|\cocov(\row(x))|\stackrel{\ref{corollaryperfectdistlattice}}{=}\cogr S_{\row(x)}\stackrel{\ref{rowmotion=AR}}{=}\cogr g_A(S_x)\stackrel{\ref{diagausregbijection}}{=}\gr S_x\stackrel{\ref{corollaryperfectdistlattice}}{=}|\cov(x)|.\]
This gives a homological proof of the classical covering theorem of Dilworth in the special case of distributive lattices, see for example \cite[Theorem 3.5.1]{KRY}.

The result of this section motivates to classify all posets whose incidence algebra is a Auslander regular algebra since by \ref{diagausregbijection} we obtain a bijection for such posets that generalises the rowmotion bijection for distributive lattices.

We give one example of a poset that is not a lattice but whose incidence algebra is diagonal Auslander regular and display the bijection obtained from \ref{diagausregbijection}.

\begin{example}
Let $P$ be the poset with the following Hasse diagram:
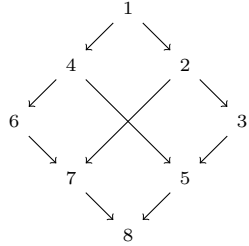
\[\begin{tikzpicture}[scale=0.5]

\node (v1) at (-19.5,27) {${\scriptstyle 1}$};

\node (v2) at (-18,25.5) {${\scriptstyle 2}$};

\node (v3) at (-16.5,24) {${\scriptstyle 3}$};

\node (v4) at (-21,25.5) {${\scriptstyle 4}$};

\node (v5) at (-18,22.5) {${\scriptstyle 5}$};

\node (v6) at (-22.5,24) {${\scriptstyle 6}$};

\node (v7) at (-21,22.5) {${\scriptstyle 7}$};

\node (v8) at (-19.5,21) {${\scriptstyle 8}$};

\draw [->] (v1) edge (v2);

\draw [->] (v1) edge (v4);

\draw [->] (v4) edge (v6);

\draw [->] (v4) edge (v5);

\draw [->] (v2) edge (v7);

\draw [->] (v2) edge (v3);

\draw [->] (v6) edge (v7);

\draw [->] (v3) edge (v5);

\draw [->] (v7) edge (v8);

\draw [->] (v5) edge (v8);
\end{tikzpicture}\]
The incidence algebra $A$ of $P$ has global dimension three and is a diagonal Auslander regular algebra, but $P$ is not a lattice.
The map $g_A$ sends $S_1$ to $S_8$, $S_2$ to $S_5$, $S_3$ to $S_4$, $S_4$ to $S_7$, $S_5$ to $S_6$, $S_6$ to $S_2$, $S_7$ to $S_3$ and $S_8$ to $S_1$.
\end{example} 

\section{Homological characterization of distributive lattices}
In \ref{homologicaldimensionsandausreg} we saw that the incidence algebra of any distributive lattice is Auslander regular.
In this section we show the converse, namely that a lattice with Auslander regular incidence algebra must be distributive.
We always assume that a poset $P$ has at least two elements and is connected in this section. Note that the assumption that the poset is connected is no loss of generality as the poset is connected if and only if the incidence algebra of $P$ is a connected algebra.
We first determine the dominant dimension of the incidence algebra of a general poset and thus which of those algebras are 1-Gorenstein.

\begin{proposition} \label{domdimtheorem}
Let $P$ be a poset with incidence algebra $A$. Then the following are equivalent:
\begin{enumerate}
\item[\rm(1)] $P$ is bounded.
\item[\rm(2)] $A$ is 1-Gorenstein.
\item[\rm(3)] There is a non-zero projective-injective module.
\end{enumerate}
If the poset is bounded, the dominant dimension is exactly one and there is exactly one indecomposable projective-injective module, namely $P(m)$ when $m$ denotes the global minimum of $P$.
\end{proposition}

\begin{proof}
Let $S_i$ denote the simple modules corresponding to the points of $P$.
We first show that (3) implies (1). Assume $P$ is not bounded. By duality we can assume there is no global maximum. Let $N=P(a)$ be the indecomposable projective module corresponding to the point $a$. Then the socle of $N$ is isomorphic to the direct sum of simple modules $S_{x_i}$, where $x_i$ are exactly those points in the poset which are larger than $a$ and are maximal, meaning that there are no larger elements in the poset than those $x_i$. Since there is no global maximum, the socle of $N$ is simple iff there is only one such $x_i$. Since indecomposable injective modules have a simple socle, in order for $N$ to be injective then, $N$ has to be isomorphic to $I(x_i)$. But $I(x_i)$ has vector space dimension larger than $N$ except when $a$ is a global minimum and in this case the socle is not simple since there is no global maximum. Thus every non-zero indecomposable projective module is not injective. Thus when $A$ has a non-zero projective-injective module, $P$ has to be bounded. \newline
Now we show that (1) implies (2). Let $P$ be bounded with global minimum $m$ and global maximum $M$.
First note that the socle of $P(m)$ is equal to $S_M$ since the unique longest path that starts at $m$ ends at $M$. Thus $P(m)$ embedds into the indecomposable injective module $I(M)$ and since both modules have the same vector space dimension, they must be isomorphic.
Then there is the projective-injective indecomposable module $P(m) \cong I(M)$ and injections $f_i: P(i) \rightarrow P(m)$ for any $i \neq m$ given by left multiplication by the unique path from $m$ to $i$. This shows that $A$ has dominant dimension at least one and that $P(m)$ is the unique indecomposable projective-injective module in that case.
That (2) implies (3) is trivial since every algebra with positive dominant dimension must contain a projective-injective module by definition. \newline
Now assume that the incidence algebra $A$ of a bounded poset $P$ has dominant dimension at least two. By the Morita-Tachikawa correspondence we have $A \cong \End_B(W)$, where $B \cong eAe$ when $e$ is an idempotent such that $eA \cong P(m)$ is the minimal faithful projective-injective module of $A$ and the $B$-module $W$ is a generator-cogenerator of $\mod B$. But then the algebra $eAe$ is isomorphic to the field $K$ because the quiver of $A$ is acyclic and $e$ is a primitive idempotent (since $A$ contains exactly one indecomposable projective-injective module) and thus with $B$ also $A$ would be semisimple as the endomorphism ring of a module over a semisimple algebra. But since we assume that $P$ is connected and has at least two elements, $A$ is not semisimple. This is a contradiction and thus the dominant dimension is equal to one.  \qedhere

\end{proof}

\begin{remark}
In \cite[Proposition 1.1]{BS}, a proof that a connected poset is a 1-Gorenstein algebra if and only if it is bounded can be found, while our result \ref{domdimtheorem} is more precise since it gives the exact value of the dominant dimension.
\end{remark}

\begin{corollary} \label{domdimlattice}
Let $P$ be a lattice. Then the incidence algebra of $P$ is 1-Gorenstein with dominant dimension precisely equal to one.
\end{corollary}

The next theorem shows that a lattice whose incidence algebra is Auslander regular has to be distributive, which completes the proof that a lattice $L$ is distributive if and only if the incidence algebra of $L$ is Auslander regular.
\begin{theorem} \label{firstmainresult}
Let $L$ be a lattice with incidence algebra $A$.
\begin{enumerate}
\item[\rm(1)] If $L$ is not distributive, then $A$ is not 2-Gorenstein.
\item[\rm(2)] $L$ is distributive if and only if $A$ is Auslander regular.
\end{enumerate}
\end{theorem}

\begin{proof}
Thanks to \ref{homologicaldimensionsandausreg}(1), it suffices to prove (1).
Assume $L$ is not distributive and has a global maximum $M$ and a global minimum $m$.
For the proof note that an algebra $A$ satisfies the 2-Gorenstein condition if and only if for every indecomposable projective module $P(i)$ we have $\pdim(J^l) \leq l$ for $l =0$ and $l=1$ when 
$$0 \rightarrow P(i) \rightarrow J^0 \rightarrow J^1 \rightarrow \cdots $$
denotes a minimal injective coresolution of $P(i)$. Also recall that in \ref{domdimtheorem} we saw that each injective envelope of an indecomposable projective module is equal to $P(m)$ and dually the projective cover of each indecomposable injective module is equal to $P(m) \cong I(M)$.
By \ref{distributivecriterion}, $L$ not being distributive means that $L$ has a sublattice isomorphic to a diamond or a pentagon. We look at both cases separately. \newline
\underline{Case 1:} Assume $L$ has a sublattice isomorphic to the diamond.
We can picture the situation as follows when looking at the Hasse quiver $Q$ of $L$: 
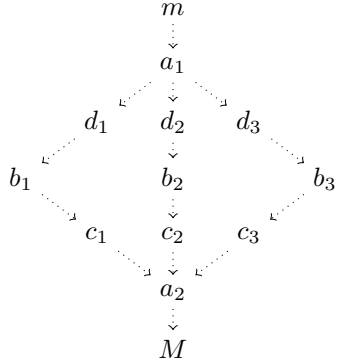
\[\begin{tikzpicture}[scale=0.5]

\node (v1) at (-16.5,23.5) {$m$};

\node (v2) at (-16.5,22) {$a_1$};

\node (v3) at (-16.5,20.5) {$d_2$};

\node (v4) at (-16.5,19) {$b_2$};

\node (v5) at (-16.5,17.5) {$c_2$};

\node (v6) at (-16.5,16) {$a_2$};

\node (v7) at (-16.5,14.5) {$M$};

\node (v8) at (-18.5,20.5) {$d_1$};

\node (v9) at (-20.5,19) {$b_1$};

\node (v10) at (-18.5,17.5) {$c_1$};

\node (v11) at (-14.5,20.5) {$d_3$};

\node (v12) at (-12.5,19) {$b_3$};

\node (v13) at (-14.5,17.5) {$c_3$};

\draw [->, dotted] (v1) edge (v2);

\draw [->, dotted] (v2) edge (v3);

\draw [->, dotted] (v3) edge (v4);

\draw [->, dotted] (v4) edge (v5);

\draw [->, dotted] (v5) edge (v6);

\draw [->, dotted] (v6) edge (v7);

\draw [->, dotted] (v2) edge (v8);

\draw [->, dotted] (v8) edge (v9);

\draw [->, dotted] (v9) edge (v10);

\draw [->, dotted] (v10) edge (v6);

\draw [->, dotted] (v2) edge (v11);

\draw [->, dotted] (v11) edge (v12);

\draw [->, dotted] (v12) edge (v13);

\draw [->, dotted] (v13) edge (v6);
\end{tikzpicture}\]
Here $a_1$ is the infimum of $b_1, b_2$ and $b_3$ and $a_2$ is the supremum of $b_1, b_2$ and $b_3$. Note that we possibly can have that $a_1=m$ and $a_2=M$.
For $i=1,2,3$, $d_i$ is the point between $a_1$ and $b_i$ that covers $a_1$ and $c_i$ is the point between $b_i$ and $a_2$ that is covered by $a_2$.
Let 
$$0 \rightarrow P(b_1) \xrightarrow{g} P(m) \rightarrow U \rightarrow 0$$
be the short exact sequence such that the left map $g$ is the injective envelope map of the projective non-injective module $P(b_1)$. The map $P(b_1) \rightarrow P(m)$ is given by left multiplication by the unique path $p_{b_1}^m$ from $m$ to $b_1$ and thus the cokernel $U$ of this map is isomorphic to $e_m A/ p_{b_1}^mA$.
Now note that the path $p_{c_2}^m$ from $m$ to $c_2$ is non-zero in $e_m A/ p_{b_1}^mA$ but $p_{c_2}^m x=0$ in $e_m A/ p_{b_1}^mA$ for any element $x$ in the radical of $A$. Thus the socle of the module $e_m A/ p_{b_1}^mA$ contains the simple module $S_{c_2}$ as a direct summand. This implies that the injective envelope of the module $e_m A/ p_{b_1}^mA$ has the indecomposable injective module $I(c_2)$ corresponding to the point $c_2$ as a direct summand. We will show that $I(c_2)$ has projective dimension at least two in order to finish the proof in case 1.
In the short exact sequence
$$0 \rightarrow K \rightarrow P(m) \xrightarrow{h} I(c_2) \rightarrow 0,$$
the right map $h$ is the projective cover of $I(c_2)$ and given by left multiplication by the element $(p_{c_2}^m)^{*}$ when $(p_{c_2}^m)^{*}$ denotes the vector space dual of the path $p_{c_2}^m$. This follows from the fact that $\top(I(c_2))=\top(D(A e_{c_2}))=D(\soc(A e_{c_2}))=D(\langle p_{c_2}^m \rangle )$.
Now the kernel $K$ of $h$ contains the path from $m$ to $d_1$, $u_1= p_{d_1}^m$, and the path from $m$ to $d_3$ $u_2=p_{d_3}^m$. This is because for $a, a' \in A$ we have: $h(e_m a)(a' e_{c_2})=(p_{c_2}^m)^{*}(e_m a a' e_{c_2})=0$ for all $a' \in A$ if $a=p_{d_1}^m$, since $p_{c_2}^m$ does not factor as a path over $u_1=p_{d_1}^m$ and similarily for $u_2$. Note that since the socle of $P(m)$ is simple, every submodule of $P(m)$ also has simple socle and thus is indecomposable.
That the kernel $K$ contains $u_1$ and $u_2$ gives us that $K$ is not a cyclic module generated by a path and thus $K$ can not be projective, since every indecomposable projective submodule of $e_m A$ is of the form $pA$ for some path $p$ and thus especially a cyclic module. As the kernel $K$ of the projective cover $h$ of $I(c_2)$ is not projective, $I(c_2)$ has projective dimension at least two. This finishes the proof for case 1. \newline
\underline{Case 2:} Now assume that $L$ has a sublattice isomorphic to a pentagon.
We can picture the situation as follows when looking at the Hasse quiver $Q$ of $L$:
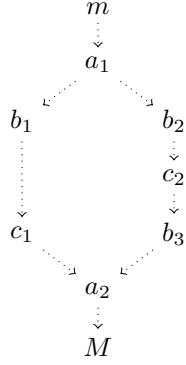
\[\begin{tikzpicture}[scale=0.5]
\node (v1) at (-16.5,23.5) {$m$};

\node (v2) at (-16.5,22) {$a_1$};

\node (v6) at (-16.5,16) {$a_2$};

\node (v7) at (-16.5,14.5) {$M$};

\node (v8) at (-18.5,20.5) {$b_1$};

\node (v10) at (-18.5,17.5) {$c_1$};

\node (v11) at (-14.5,20.5) {$b_2$};

\node (v12) at (-14.5,19) {$c_2$};

\node (v13) at (-14.5,17.5) {$b_3$};

\draw [->, dotted] (v1) edge (v2);

\draw [->, dotted] (v6) edge (v7);

\draw [->, dotted] (v2) edge (v8);

\draw [->, dotted] (v8) edge (v10);

\draw [->, dotted] (v10) edge (v6);

\draw [->, dotted] (v2) edge (v11);

\draw [->, dotted] (v11) edge (v12);

\draw [->, dotted] (v12) edge (v13);

\draw [->, dotted] (v13) edge (v6);
\end{tikzpicture}\]
Here $a_1$ is the infimum of $b_1$ and $b_2$ and $a_2$ is the supremum of $b_1$ and $b_3$. Note that we possibly can have $m=a_1$ and $M=a_2$.
$c_2$ is covered by $b_3$ and is between $b_2$ and $b_3$ and $c_1$ covers $b_1$ and is between $b_1$ and $a_2$.
Let 
$$0 \rightarrow P(b_3) \xrightarrow{g} P(m) \rightarrow U \rightarrow 0$$
be the short exact sequence such that $g$ is the injective envelope of $P(b_3)$. Then $g$ is given by left multiplication with the unique path $p_{b_3}^m$ from $m$ to $b_3$ and thus $U= e_m A/ p_{b_3}^m A$. Now the path $p_{c_2}^m$ from $m$ to $c_2$ is nonzero in $e_m A/ p_{b_3}^m A$ and $p_{c_2}^m x=0$ in $e_m A/ p_{b_3}^m A$ for every $x$ in the radical of $A$ and thus the socle of $U$ contains $S_{c_2}$. This shows that the injective envelope of $e_m A/ p_{b_3}^m A$ contains $I(c_2)$ as a direct summand. We show that the projective dimension of the indecomposable injective module $I(c_2)$ corresponding to the point $c_2$ is at least two.
Let 
$$0 \rightarrow K \rightarrow P(m) \xrightarrow{h} I(c_2) \rightarrow 0$$
be the exact sequence such that $h$ is the projective cover of $I(c_2)$. Then $h$ is given by left multiplication with the vector space dual of the the path $p_{c_2}^m$. We again show that $K$ is not a cyclic module to finish the proof.
In order to see this just note that the kernel $K$ of $h$ contains the paths $p_{b_3}^m$ from $m$ to $b_3$ and $p_{b_1}^m$ from $m$ to $b_1$ and as in case 1 one concludes that $K$ can not be cyclic and thus not be projective. This finishes the proof of case 2. 
\end{proof}

We give an example showing that there are posets that contain a pentagon as a subposet but whose incidence algebra is Auslander regular. This shows that the proof of \ref{firstmainresult} can not be extented to general bounded posets.
\begin{example} \label{posetexample1}
The following poset $P$ with Hasse quiver $Q$ is a bounded poset, which contains a pentagon as a subposet and is not a lattice. Its incidence algebra $A$ is Auslander regular and has global dimension equal to 3.
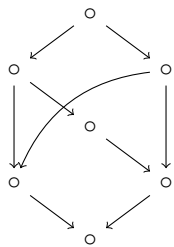
\[\begin{tikzpicture}[scale=0.5]
\node (v1) at (-19.5,24.5) {$\circ$};

\node (v2) at (-21.5,23) {$\circ$};

\node (v3) at (-19.5,21.5) {$\circ$};

\node (v4) at (-17.5,23) {$\circ$};

\node (v5) at (-17.5,20) {$\circ$};

\node (v6) at (-21.5,20) {$\circ$};

\node (v7) at (-19.5,18.5) {$\circ$};

\draw [->] (v1) edge (v2);

\draw [->] (v1) edge (v4);

\draw [->] (v2) edge (v3);

\draw [->] (v2) edge (v6);

\draw [->] (v4) [bend right=30] edge (v6);

\draw [->] (v4) edge (v5);

\draw [->] (v3) edge (v5);

\draw [->] (v6) edge (v7);

\draw [->] (v5) edge (v7);

\end{tikzpicture}\]
\end{example}

\section*{Acknowledgments} 
The first author is supported by JSPS Grant-in-Aid for Scientific Research (B) 16H03923, (C) 18K3209 and (S) 15H05738.
Rene Marczinzik is funded by the DFG with the project number 428999796.
We profited from the use of the GAP-package \cite{QPA} and \cite{SAGE}.

\end{document}